\documentclass[12pt]{elsarticle} %[reqno]{amsart}
\usepackage{graphicx}
\usepackage{amsmath}
\usepackage{bm}
\usepackage{amssymb}
\usepackage{amsfonts}
\usepackage{amsthm}
\usepackage{placeins}
\usepackage{afterpage}
\newtheorem{Theorem}{Theorem}[section]
\newtheorem{Lemma}{Lemma}[section]

\def\ee{\end{eqnarray*}}
\def\be{\begin{eqnarray*}}
\def\bee{\end{eqnarray}}
\def\bbe{\begin{eqnarray}}
\def\ea{\end{align*}}
\def\ba{\begin{align*}}
\makeatletter
\let\today\relax
\def\ps@pprintTitle{%
    \let\@oddhead\@empty
    \let\@evenhead\@empty
    \def\@oddfoot{\footnotesize\itshape
      \hfill\today}%   {Preprint submitted} \hfill\today}%
    \let\@evenfoot\@oddfoot
    }
\makeatother

\makeatletter
\newcommand{\bal}{\@ifstar{\@bals}{\@bal}}
\def\@bals#1\eal{\begin{align*}#1\end{align*}}
\def\@bal#1\eal{\begin{align}#1\end{align}}
\makeatletter
\def\A{A} %\mathbf{A}}
\def\AA{\mathcal A}
\def\B{B}
\def\L{\mathcal{L}}
\def\K{\mathcal K}
\def\g{ g}
\def\v{v} %\bm v}
\def\u{u}%\bm u}

    \def\p{\partial}

\DeclareMathOperator*{\esssup}{ess\,sup}

\makeatletter
\newcommand{\subjclass}[2][1991]{%
  \let\@oldtitle\@title%
  \gdef\@title{\@oldtitle\footnotetext{#1 \emph{Mathematics subject classification.} #2}}%
}
\makeatother

\begin{document}

\title{Strong solutions  of fractional Boussinesq equations in an exterior domain  }

\author{Zhi-Min Chen}
 \ead{zmchen@szu.edu.cn}
\author{Qiuyue Zhang}%\ead{qyzhang0722@163.com}

\address{School  of Mathematical Sciences, Shenzhen University, Shenzhen 518060,  China}%

%\date{}% It is always \today, today,

\begin{abstract} A thermal convection fluid motion in the three-dimensional  domain exterior to a sphere  is considered. A purely conductive steady state arises due to the fluid heated from the sphere.  A fractional equation system is introduced by using spectral presentation. The existence of small strong solutions  in a Hilbert space is obtained. The strong solution existence  implies  the local stability of the steady state, which attracts asymptotically  the flows  evolving initially  from the vector fields close to the steady state.

\end{abstract}
%\subjclass[2020]{35D30, 35Q30, 35Q35, 76D05}

\maketitle

\section{Introduction}
Consider a large-scale atmospheric  fluid motion  in the exterior domain $\Omega = \{ x \in R^3; \,\, |x|>1\}$ around an earth surface $\p\Omega = \{ x \in R^3; \,\, |x|=1\}$, which is heated at a temperature $\alpha>0$. The motion is in a  non-dimensional form and the temperature at the infinity is assumed to be zero.  Thus  flow is driven by a buoyant force determined by  the temperature  $\alpha$  and the gravity
$$ g(x) = g_0 \nabla \frac1{|x|}$$
with a gravitational constant $g_0$.  Here, for simplicity, $g_0=1$ is assumed and   the sphere temperature $\alpha$ is  accepted  to be a constant, which  gives rise to a purely conductive steady state
with zero fluid velocity. %However, when $\alpha$ increases, the conductive steady state may  lose the  stability and  develop into  turbulence.

In this paper, we are interested in the local stability of the purely conductive steady state when initial fluid velocity and temperature are  close to the state by considering the existence of small  strong solutions in a Hilbert space.
This  fluid motion in the exterior domain $\Omega$  is  governed by the following Boussinesq equation system
\bbe \left.\begin{array}{ll}
\p_t \u + \u\cdot \nabla \u = \Delta \u -\nabla p -\alpha    \g T,\,\, \nabla \cdot \u=0,\,\, &\label{a1}
\\ \p_t T +\u\cdot \nabla T =\Delta T,\,\,&
\\
 \u|_{\p\Omega}=0,\,\, T|_{\p\Omega}=\alpha  . &
\end{array}\right\}
\bee
Here  the time-dependent functions $\u$, $T$ and $p$ represent, respectively,  the  unknown velocity, temperature and pressure.

The  purely convective steady-state solution of (\ref{a1}) is denoted by   $(\hat u,\hat p, \hat T)$, which is subject to the equations
\bbe \hat u=0,\,\,\,  \nabla \hat p= -\alpha    \g \hat T,\,\, \Delta \hat T=0, \,\,\hat T|_{\p\Omega}=\alpha   ,
\bee
and hence is  expressed explicitly as
\bbe \hat u=0,\,\,\, \hat T=\frac \alpha    {|x|}, \,\, \hat p(x) = -\frac{ \alpha   ^2}{ 2 |x|^2}.
\bee
%Here, for the simplicity of notion,  the gravitational constant  $g_0$ is assumed to be $1$. Otherwise, we may absorb $g_0$ into $\alpha $.
 Local stability of the purely steady-state flow was studied by Hishida and Yamada \cite{HY} and Hishida \cite{H94,H97} with respect to strong solutions, while $L^2$ global stability of the purely steady-state flow has been studied by Chen-Kagei-Miyakawa \cite{Chen1992} with respect to weak solutions.

Let $L^2_\sigma(\Omega)^3$ be the $L^2$-closure of the set $C_{0,\sigma}(\Omega)^3$ of all smooth
solenoidal vector fields with compact supports in  $\Omega$.  Let  $P$ denote the bounded projection operator mapping $L^2(\Omega)^3$ onto  $L^2_\sigma(\Omega)^3$ due to the Helmholtz decomposition \cite{Miya1988}.
By adopting  the perturbation
\bbe (\u, p, T)= (\hat u,\hat p,\hat T)+ (u',p',T')\label{T0}
\bee
 and omitting the superscript primes, equation (\ref{a1}) is   written as
\bbe \left.\begin{array}{ll}
\p_t \u -\Delta \u+\nabla p+a   \g T=-\u\cdot \nabla \u,  &\label{a2}
\\ \p_t T -\Delta T+\u\cdot \nabla \hat T=-\u\cdot \nabla T,&
\\
 \u|_{\p\Omega}=0,\,\, T|_{\p\Omega} =0.
\end{array}\right\}
\bee

To simplify the problem,  the linearized operator of (\ref{a2}) is expressed as
\bbe \L U  \doteq ( -P\Delta \u+Pa   \g T, -\Delta T+\u\cdot \nabla  \hat T ),\,\, U=(u,T), \bee
with the domain
\bbe  D(\L ) \doteq D(\A) \times D(\B)
\bee
for the Stokes operator part
\bbe D(\A)\doteq \left\{ u \in W^{2,2}(\Omega)^3; \nabla\cdot u=0, u|_{\p\Omega}=0\right\},
\bee
and the Laplacian part
\bbe D(\B)\doteq \left\{ T \in  W^{2,2}(\Omega);  T|_{\p\Omega}=0\right\}.
\bee
The operator $\L$ is  similar to the Laplacian in the sense of  the self-adjointness
\bbe \langle \L U  , V\rangle
 &=& \langle -\Delta  \u+ \alpha T\nabla \frac1{|x|},\v\rangle +\langle -\Delta T+\alpha\u\cdot \nabla \frac1{|x|},\tilde T\rangle\nonumber
\\
&=& \langle \u, -\Delta  \v+ \alpha  \tilde T\nabla\frac1{|x|} \rangle   +\langle  T, -\Delta \tilde T +\alpha \v\cdot \nabla \frac1{|x|}\rangle\nonumber
\\
&=& \langle U ,\L V\rangle ,\label{ne1}
\bee
and the positivity
\bbe \langle \L  U,U\rangle &=&\langle A   \u+ \alpha  T\nabla\frac1{|x|},\u\rangle +\langle -\Delta T+\alpha\u\cdot \nabla \frac1{|x|},T\rangle \nonumber
%\\ &=&\|\nabla   \u\|_2^2+\|\nabla  T\|_2^2+2\alpha \langle \u\cdot \nabla \frac1{|x|},T\rangle \nonumber
\\ &=&\|\nabla   \u\|_2^2+\|\nabla  T\|_2^2-2\alpha \langle  \frac\u{|x|},\nabla T\rangle \label{pos}
\\ &\ge &\|\nabla  \u\|_2^2+\|\nabla  T\|_2^2-4\alpha   \|\nabla u\|_2\|\nabla T\|_2\nonumber
\\ &\ge &\left(1-2\alpha \right) \|\nabla  U\|_2^2,\label{positive}
\bee
for $1-2\alpha >0$,  $U=(u,T) ,\, V=(v,\tilde T)\in D(\L )$, $\|\cdot \|_r $ the norm of the Lebesgue space $L^r$ over the domain $\Omega$, and   the $L^2$ inner product $\langle \cdot, \cdot\rangle= \|\cdot \|_2^2$.
Here we have used integration by parts, H\"older inequality  and  the estimate (Ladyzhenskaya \cite[page 41]{Lady})
\bbe \int_\Omega \frac{|u(x)|^2}{|x|^2} dx \le 4\int _\Omega |\nabla u(x)|^2 dx. \label{Lady}\bee
Therefore, equation (\ref{a2}) is rewritten as
\bbe \p_tU+ \L  U= f(U)
\bee
for \bbe  f(U)\doteq (-P(u\cdot \nabla u), -u\cdot \nabla T),\,\,\, U=(u,T)\emph{\emph{}}.
\bee
More generally,  we consider the fractional Boussinesq equation
\bbe \p_t U + \L ^\kappa  U = f(U), \label{ne2}
\\
U|_{t=0} =U_0 \doteq (u_0,T_0),\label{ne3}
\bee
by adopting the spectral decomposition (see Yosida \cite[Page 313, Theorem 1]{Yosida})
\bbe
\L ^\kappa  = \int^\infty_0 \lambda^\kappa   dE_\lambda
\bee
due to the self-adjointness  (\ref{ne1}) and the positivity (\ref{positive}). Here $E_\lambda$ denotes  the spectral resolution of the unit determined by the  operator $\L $. The integral formulation of (\ref{ne2})-(\ref{ne3}) is
\bbe U(t)= e^{-t\L^\kappa  } U_0 + \int^t_0 e^{-(t-s)\L ^\kappa }f(U(s)) ds.\label{ne4}
\bee
To state the main result, we need the fractional operator
\bbe \AA^\kappa (u,T) = (\A^\kappa u, B^\kappa T)
\bee
for the fractional Stokes operator $\A^\kappa$ and fractional Laplace operator $B^\kappa$ defined respectively  by their corresponding spectral representations as both $A$ and $B$ are self-adjoint and positive. By the Sobolev imbedding
 \bbe \|u\|_6 \le C \|\nabla u\|_2, \|T\|_6 \le C \|\nabla T\|_2,\label{Sobolev}
 \bee
 for $(u,T)\in D(\AA)=D(\A)\times D(\B)$,
 we define $H^\beta$ the completion of $D(\AA)$ in the norm
$\|\AA^\beta (u,T)\|_2$ for $0\le \beta \le \frac12.$

We are in the position to state the main result.
\begin{Theorem}\label{Th1} For   $\frac 34 < \kappa  \le 1$, equation    (\ref{ne4}) admits a unique solution $ U=(u,T) \in C([0,\infty); H^{\frac54-\kappa })$
 subject to the following properties:
\bbe \esssup_{0<t<\infty} \left(\|\AA ^{\frac54-\kappa } U(t)\|_{2}+t^{1-\frac3{4\kappa }} \|\nabla U(t)\|_{2}\right)<\infty,
\bee
and
\bbe \lim_{t\to\infty}\left( \|\AA ^{\frac54-\kappa }U(t)\|_2 +t^{1-\frac3{4\kappa }} \|\nabla U(t)\|_{2}\right)=0,
\bee
provided the initial data $U_0 \in H^{\frac54-\kappa }$  and $\|\AA ^{\frac54-\kappa } U_0\|_{2}$ is  sufficiently small.

\end{Theorem}

To derive the stability, we assume that $\alpha >0$ is sufficiently small, so that the buoyant force and external influence from nonhomogeneous boundary condition can be controlled by the viscous force and heat diffusion.
If the temperature on the conducting surface $\p\Omega$ is not small, it is infeasible to keep the stability. For example, instabilities  arise in  the Rayleigh-B\'enard thermal convection flow (see Lorenz \cite{L}, Chen and Price \cite{Chen2006} and Rabinowitz \cite{R}).
When $\kappa =1$, the existence of small strong solution in a Hilbert space  was studied  by Hishida \cite{H94}.

\section{Preliminaries}
Let $C$  be a generic constant, which is   independent of the quantities $U$, $u$, $T$,$V$, $v$, $T' $,$\tau$, $\tau_1$, $t>0$ and  $x\in \Omega$, but they may depend on initial vector field $U_0=(u_0,T_0)$.

\begin{Lemma} \label{LL1} Let $\K$ be a self-adjoint and positive operator in a Hilbert space $H$ and thus the fractional power operator
is defined as
\bbe \K^\beta = \int^\infty_0 \lambda^\beta d\hat E_{\lambda}\bee
with the domain
$  D(\K^\beta)$ under the graph norm
\bbe\|\phi\|_{D(\K^\beta)} = \|\phi\|_H+ \|\K^\beta\phi\|_H.\bee
Here $\hat E_{\lambda}$ denotes the spectral resolution of the unit determined by $\K$.
For $0\le \beta \le 1$, assume that $\hat H^\beta$,  the completion of $D(\K)$  in $\|\K^\beta \phi\|_H$, is a Banach space.
 Then, there holds the topological  identity relation
\bbe [ H, \hat H^\beta ] _\theta = \hat H^{\beta \theta}, \,\,0<\theta <1 \label{complex}
\bee
for $[\cdot,\cdot ]_\theta$ the complex interpolation functor.
\end{Lemma}

When $\K$ is the Stokes operator $A$ and $\beta =\frac12$, this result has been proved by Miyakawa \cite[Theorem 2.4]{Miya82}. However, the proof of \cite[Theorem 2.4]{Miya82} also applies  for the operator $\K$. We thus omit the proof herein.  One may also refer to \cite[Theorem11.6.1]{in} for the topological identity \bbe [H, D(\K^\beta)]_\theta = D(\K^{\beta\theta}).\bee

\begin{Lemma} For $0\le \kappa \le \frac12$, we have that $H^\kappa$ is the completion of $D(\L)$ under the norm
$\|\L ^\kappa U\|_2$. That is,
\bbe \|\AA ^\kappa U \|_2 \le C \|\L^\kappa U \|_2,\,\,\, \|\L ^\kappa U \|_2 \le C \|\AA^\kappa U \|_2. \label{bound}\bee

\end{Lemma}

\begin{proof} By (\ref{positive}) and since  $U\in D(\L)$, we have
\bbe \|\nabla U\|_2 \le C\|\L^\frac12 U\|_2.
\bee
By (\ref{pos}) and (\ref{Lady}), we have
\bbe \|\L^\frac12 U\|_2 \le C\|\nabla  U\|_2.
\bee
Denoting  by $\tilde H^\beta$ the completion of $D(\L)$ under the norm $\|\L^\beta U\|_2$ for $0\le \beta\le \frac12$, we thus have
$\tilde H^\frac12 =H^\frac12$. Moreover, since $H^0=\tilde H^0= L^2_\sigma(\Omega)^3\times L^2(\Omega)$, it follows from Lemma \ref{LL1} that
\bbe H^{\frac\theta2}= [H^0, H^\frac12]_\theta = [\tilde H^0,\tilde H^\frac12]_\theta= \tilde H^{\frac\theta2}. \bee
This gives the validity of  (\ref{bound}).
\end{proof}

%The analysis is essentially based on the following estimates.
\begin{Lemma} For $\frac12\le  \kappa \le 1$,  $0<\beta \le \frac1{2\kappa }$, $ 0\le \gamma  \le \kappa$,  $0<\delta <  \frac12$ and  $\frac12\le \frac1r <\frac 56$, we have
\begin{align}
 \|\nabla  e^{-t \L ^\kappa  } U\|_{2} &\le   C t^{-\beta  } \|\L  ^{\frac12-\kappa  \beta}U\|_{2}, \,\, U\in H^{\frac12-\kappa \beta},\label{aa2}
 \\
 \|\L^\gamma  e^{-t \L ^\kappa  } U\|_{2} &\le   C t^{-\frac\gamma\kappa  } \|U\|_{2}, \,\, \,\, U\in L^2_\sigma(\Omega)^2\times L^2(\Omega),\label{aa1aa}
  \\
 \|\L ^{\gamma  } e^{-t \L ^\kappa  } \!\!(Pu,T)\|_{2} &\le   t^{-\frac\gamma {\kappa }-\frac3{2\kappa } (\frac1r-\frac12) } \|(u,T)\|_r,  \,\, (u,T)\in L^r(\Omega)^3\!\times \!L^r(\Omega),\label{aa1}
\\
  \lim_{t\to 0} t^{\frac{1-2\delta}{2\kappa }}\|\nabla  e^{-t\L^\kappa }U\|_{2}&=0,\,\,\,\,  U\in H^{\delta},\label{aa3new}
  \\
   \lim_{t\to \infty} \|e^{-t\L^\kappa}U\|_2 &=0,\,\,\,\, U\in L^2_\sigma(\Omega)^2\times L^2(\Omega).\label{infinity}
\end{align}

\end{Lemma}
\begin{proof}
The  semigroup  $e^{-t\L^\kappa }$ being analytic  is confirmed by (\ref{aa1aa}) with $\gamma=\kappa $.
Equation (\ref{aa2}) is
  given   by  (\ref{positive})  together with the spectral representation:
\bbe
\|\nabla e^{-t \L ^\kappa } U\|_2 &\le& C \| \L ^{\frac12} e^{-t \L ^\kappa }U\|_2\nonumber
\\
 &\le &C t^{-\beta}\|\left(\int^\infty_0 (t\lambda^\kappa  )^\beta e^{-t \lambda ^\kappa  } dE_\lambda\right) \L ^{\frac12-\kappa  \beta}U \|_2.
\bee
By   the spectral representation, we obtain   (\ref{aa1aa}) from the following
\bbe\label{nne1} \L ^{\gamma} e^{-t \L ^\kappa } &=&
t^{-\frac\gamma{\kappa }}\int^\infty_0 (t\lambda^\kappa  )^\frac\gamma{\kappa } e^{-t \lambda ^\kappa  } dE_\lambda.\bee

  For (\ref{aa1}) with $\gamma=0$ and $r\ne 2$, we employ  H\"older inequality,   Sobolev imbedding (\ref{Sobolev}),   (\ref{aa2}) and  (\ref{aa1aa}), and take  $\frac 1q=1-\frac1r$  and $\theta =\frac52-\frac3r$ to produce % or $\frac1q=\frac T3+\frac16$.
\bbe \langle  e^{-t \L ^\kappa } (Pu,T),V\rangle &=& \langle  (u,T), e^{-t \L ^\kappa } V\rangle\nonumber
\\
&\le&
 \|(u,T)\|_r \| e^{-t \L ^\kappa } V\|_q\nonumber
\\
&\le &  C\|(u,T)\|_r \| e^{-t \L ^\kappa } V\|_2^\theta\|\nabla  e^{-t \L ^\kappa } V\|_2^{1-\theta}\nonumber
\\
&\le& C  t^{-\frac{1-\theta}{2\kappa}} \|(u,T)\|_r\|V\|_2 %\le t^{-\frac{3(\frac1r-\frac12)}{2\kappa } }\|(u,T)\|_r \|  V\|_2
\bee
for $V \in L_\sigma^2(\Omega)^2\times L^2(\Omega)$. Therefore,  since $1-\theta = 3(\frac1r-\frac12)$, we have
\bbe \|e^{-t\L^\kappa  } (Pu,T)\|_2
\le Ct^{-\frac{3}{2\kappa }(\frac1r-\frac12)} \|(u,T)\|_r\label{aa1aaa}
\bee
 Thus the combination of (\ref{aa1aa}) and (\ref{aa1aaa}) gives  (\ref{aa1}) due to the semigroup property of $e^{-t\L^\kappa }$.

Moreover, to show (\ref{aa3new}), we take  $V\in D(\L )$ and  employ (\ref{aa2})-(\ref{aa1}) to produce
\bbe\nonumber  t^{\frac{1-2\delta}{2\kappa }}\|\nabla  e^{-t \L ^\kappa }U\|_{2}&\le & t^{\frac{1-2\delta}{2\kappa }}\|\L ^{\frac12-\delta}  e^{-t \L ^\kappa } (\L ^\delta U-\L ^\delta V)\|_{2}+  t^{\frac{1-2\delta}{2\kappa }}\|  e^{-t \L ^\kappa } \L ^\frac12 V\|_{2}
\\
&\le & \|\L ^\delta U-\L ^\delta V\|_{2}+  t^{\frac{1-2\delta}{2\kappa }}\|  \L ^\frac12 V\|_{2}.\label{newaa}
\bee
This gives (\ref{aa3new}),  after taking $t\to 0$ and using the density of $D(\L )$ in $H^\delta$.

Finally, for the proof of (\ref{infinity}), we employ (\ref{aa1aa}) and the derivation of (\ref{aa1aaa}) to obtain
\bbe
\|e^{-t L^\kappa} U\|_2&\le &\|e^{-t L^\kappa} (U-V)\|_2+\|e^{-t L^\kappa} V\|_2\nonumber
\\
&\le &\|U-V\|_2+Ct^{-\frac1{4\kappa}}\|V\|_{\frac32}
\bee
for $V \in C^\infty_{0,\sigma}(\Omega)^3 \times C^\infty_0(\Omega)$, which is dense in $H^0$. Letting $t\to \infty$ in the previous equation, we obtain (\ref{infinity}) due to the density.
\end{proof}

\section{Proof of the Theorem 1.1}

To prove Theorem \ref{Th1} or  the small solution existence, we adopt the operator
\bbe
FU(t)\doteq e^{-t\L^\kappa  } U_0 -\int^t_0 e^{-(t-s)\L ^\kappa  } f(U) ds.\label{new1}
\bee
We shall use the Banach contraction mapping principle to show the existence of the unique fixed point of $F$ in the complete metric space
\be
X=&& \left\{
  U\in C([0,\infty); H^{\frac54-\kappa }); \,\,\lim_{t\to 0} t^{1-\frac3{4\kappa }} \|\nabla U(\tau +t)\|_{2}=0\,\,\mbox{ for } \tau\ge 0, \right.
  \\
&&\hspace{4mm} \|U\|_X =\esssup_{t>0} \left(t^{1-\frac3{4\kappa }} \|\nabla U(t)\|_{2} + \|\AA ^{\frac54-\kappa } U(t)\|_{2}\right) \le M,
\\
&&\hspace{4mm} \left. \lim_{t\to \infty}\left(t^{1-\frac3{4\kappa }} \|\nabla U(t)\|_{2} + \|\AA ^{\frac54-\kappa } U(t)\|_{2}\right) =0\right\},
\ee
where  the bound $M$  is to be determined afterwards.

      To estimate the nonlinear integrand of (\ref{new1}), we employ (\ref{bound}), (\ref{aa1}),  H\"older inequality and (\ref{Sobolev})  to produce, for $0\le \beta\le \frac12$ and  $U=(u,T)\in X$,
   \bbe \nonumber\|\AA ^{\beta}e^{-t \L ^\kappa  } f(U)\|_2 &\le& \|\L ^{\beta}e^{-t \L ^\kappa  } f(U)\|_2
   \\ \nonumber
   &\le &Ct^{-\frac{\beta}{\kappa }-\frac3{2\kappa }(\frac23-\frac12)}(\|u\cdot \nabla u\|_{\frac32}+ \|u\cdot \nabla T\|_{\frac32})
    \\ \nonumber
   &\le &Ct^{-\frac\beta\kappa  -\frac1{4\kappa }}\|u\|_6 \| \nabla U\|_{2}
   \\ \label{intd}
   &\le &Ct^{-\frac\beta\kappa  -\frac1{4\kappa }}\|\nabla u\|_{2}\|\nabla U\|_2.
   \bee
 Hence we have, by (\ref{aa1aa}),  (\ref{intd})  and the condition  $\frac 34<\kappa  \le 1$,
\bbe \nonumber
\lefteqn{\|\L ^{\frac54-\kappa }FU(t)\|_{2}}\\
&\le &\|\L ^{\frac54-\kappa } e^{-t \L ^\kappa  } U_0\|_{2}+ C\int^t_0 (t-s)^{-(\frac{3}{2\kappa }-1)} \| \nabla u(s) \|_{2}\|\nabla U(s)\|_2ds\nonumber
\\
&\le & \|\L ^{\frac54-\kappa } U_0\|_{2}\nonumber
+ CM\int^t_0   (t-s)^{-(\frac{3}{2\kappa }-1)} s^{-2(1-\frac3{4\kappa })}ds \esssup_{0<s<t} s^{1-\frac3{4\kappa }}\|\nabla u(s)\|_{2}\nonumber
\\
&\le & \|\L^{\frac54-\kappa } U_0\|_{2}+  CM\esssup_{0<s<t} s^{1-\frac3{4\kappa }}\|\nabla u(s)\|_{2}\label{kkkk}
\bee
and, by (\ref{positive}) and (\ref{intd}),
\bbe
  \|\nabla FU(t)\|_{2}\nonumber
&\!\!\le\!\!&  \|\L ^{\frac12} e^{-t \L ^\kappa  }U_0\|_{2}+ C\int^t_0 (t-s)^{-\frac3{4\kappa }}  \| \nabla u(s) \|_{2}\|\nabla U(s)\|_2ds\nonumber
\\
&\!\!\le\!\!&  t^{-(1-\frac3{4\kappa })} \|\L ^{\frac12-\kappa (1-\frac3{4\kappa })} U_0\|_{2}\nonumber
\\
&&+ CM\int^t_0 (t-s)^{-\frac3{4\kappa }}  s^{-2(1-\frac3{4\kappa })}ds \esssup_{0<s<t} s^{1-\frac3{4\kappa }}\|\nabla u(s)\|_{2}\nonumber
\\
&\!\!\le\!\!& \!\! t^{-(1-\frac3{4\kappa })} \|\L ^{\frac54-\kappa } U_0\|_{2}\!+\! CM  t^{-(1-\frac3{4\kappa })}\esssup_{0<s<t} s^{1-\frac3{4\kappa }}\|\nabla u(s)\|_{2}.\hspace{5mm}\label{kkk}
\bee
Hence, by (\ref{bound}), we have
\bbe \|Fu\|_X \le  \|\AA ^{\frac54-\kappa } U_0\|_{2}+ C M^2\le M,
\bee
provided that  the bound $M$ is sufficiently small and
\bbe \|\AA ^{\frac54-\kappa } U_0\|_{2}\le \frac M2.\bee

Moreover, for the continuity at $\tau\ge 0$ and the zero limit, we see that
 \bbe FU(t+\tau) = e^{-t\L^\kappa } FU(\tau) - \int^t_0 e^{-(t-s)\L ^\kappa  } f( U(s+\tau)) ds\label{tau}
 \bee
due to the semigroup property. %Thus it is sufficient to  consider the continuity and  the zero limit  at $\tau=0$.
Therefore, by (\ref{bound}),  (\ref{aa3new}), the derivation of (\ref{kkkk})-(\ref{kkk})  and the strong continuity of the analytic semigroup,  we have
\be
\lefteqn{\|\AA ^{\frac54-\kappa }(FU(t+\tau)-FU(\tau))\|_{L_2}}\\
&\le & C\|\L ^{\frac54-\kappa }(FU(t+\tau)-FU(\tau))\|_{L_2}
\\
&\le& \| (e^{-t \L ^\kappa  }-1) \L ^{\frac54-\kappa } FU(\tau)\|_{2}+ C M\esssup_{0<s<t} s^{1-\frac3{4\kappa }}\|\nabla u(s+\tau)\|_{2} \to 0,
\ee
\be
 t^{1-\frac3{4\kappa }}\|\nabla FU(t)\|_{2}
&\le & t^{1-\frac3{4\kappa }}\|\nabla  e^{-t\L^\kappa  }U_0\|_{2}+ C M\esssup_{0<s<t} s^{1-\frac3{4\kappa }}\|\nabla u(s)\|_{2} \to 0,
\ee
as $t\to 0$.

To show  the decay property, for any $\epsilon >0$, there exist $\tau>0$  and $\tau_1>\tau$ so that
\bbe  t^{1-\frac3{4\kappa }} \|\nabla u(t)\|_{2}  \le \epsilon,\,\, t \ge \tau,
\\
\| e^{-\frac12t \L ^\kappa  } \L ^{\frac54-\kappa } FU(\tau)\|_{2}\le \epsilon,\,\, t \ge \tau_1
\bee
due to   (\ref{infinity}) and the condition $U\in X$.
 Following  the  derivation of (\ref{kkk}) and (\ref{kkkk}) in (\ref{tau}), we have
 \be
 \lefteqn{\|\AA ^{\frac54-\kappa }FU(t+\tau)\|_2 +(t+\tau)^{1-\frac3{4\kappa }} \|\nabla FU(t+\tau)\|_{2}}
 \\
 &\le&  C\| e^{-\frac12t \L ^\kappa  } \L ^{\frac54-\kappa } FU(\tau)\|_{2}+ CM (1+\frac\tau t)^{1-\frac3{4\kappa }}\esssup_{0<s<t}s^{1-\frac3{4\kappa }} \|\nabla u(s+\tau)\|_{2}
 \\
 &\le & C \epsilon +CM (1+\frac\tau{\tau_1})^{1-\frac3{4\kappa }}\epsilon.
 \ee
 The small constant $M$ can be assumed as $M\le 1$. Thus we have, for any small $\epsilon >0$,
 \bbe \lim_{t\to\infty}\left( \|\AA ^{\frac54-\kappa }FU(t)\|_2 +t^{1-\frac3{4\kappa }} \|\nabla FU(t)\|_{2}\right)\le C\epsilon. \bee
 This gives the desired decay property.
Collecting terms, we have the injection property $F: X\mapsto X$.

Similarly, for $U,\, V\in X$,
we have the contraction property
\be \|FU-FV\|_X &\le& C(\|U\|_X +\|V\|_X) \|U-V\|_X
\\
&\le& C M \|U-V\|_X\le \frac12 \|U-V\|_X,
\ee
provided that $M$ is sufficiently small.
We thus have the desired  unique solution existence due to the Banach contraction mapping principle. The proof of Theorem \ref{Th1} is complete.

\

%%%%%%%%%%%%%%%%%%%%%%%%%%%%%%%%%%%%%%%%%%%%%%%%%
\noindent {\bf Acknowledgment.}
This research was supported by The Shenzhen Natural Science Fund of China (Stable Support Plan Program 20220805175116001).

\

%\noindent {\bf Conflict of interest}  The author states that there is no conflict of interest.

%\

%\noindent {\bf Data availability} My manuscript has no associated data.


\begin{thebibliography}{99}

%\bibitem{Miya1990} W. Borchers and T. Miyakawa,
%Algebraic $L^2$ decay for Navier-Stokes
%flows in exterior domains.  Acta Math. 165 (1990), 189-227.

%\bibitem{Miya1992} W. Borchers and T. Miyakawa, $L^2$-Decay for Navier-Stokes Flows
%in Unbounded Domains,
%with Application to Exterior Stationary Flows. Arch. Rational Mech. Anal. 118 (1992), 273-295.

%\bibitem{Miya1995} W. Borchers and T. Miyakawa, On stability of exterior
%stationary Navier-Stokes flows. Acta Math. 174 (1995), 311-382.

\bibitem{in} C. M. Carracedo and M. S. Alix, The Theory of Fractional Powers of Operators, North-Holland Mathematics Studies Vol 187, 2000, North-Holland,

%\bibitem{Chen2012} Z. -M. Chen, A vortex based panel method for potential flow simulation
%around a hydrofoil, J. Fluid Structure 28 (2012), 378-391
%\bibitem{Chen2019} Z. -M. Chen, Straightforward integration for free surface Green function and body
%wave motions, European J Mech./B Fluids 74(2019) 10-18


\bibitem{Chen1992}Z. -M. Chen, Y. Kagei and T. Miyakawa, Remarks on stability of purely conductive
steady states to the exterior Boussinesq problem, Adv. Math. Sci. Appl. 1 (1992), 411-430.

%\bibitem{Chen1993} Z. -M. Chen, Solutions of the stationary
%and nonstationary
%navier-stokes equations in exterior domains, Pacific J. Math. 159 (1993), 227-240.


\bibitem{Chen2006} Z. -M. Chen and  W. G. Price, On the relation between Rayleigh-B\'{e}nard
convection and Lorenz system.
 Chaos, Solitons and Fractals 28 (2006), 571-578.


%\bibitem{Han} P. Han, Decay rates for the incompressible Navier-Stokes flows in 3D exterior domains. J. Diff. Equs.
%263 (2012), 3235-3269.

%\bibitem{He} C. He and T. Miyakawa, On weighted-norm estimates for nonstationary incompressible Navier-Stokes flows in a 3D exterior domain. J. %Diff. Equs.
% 246 (2009), 2355-2386.

 \bibitem{H94} T. Hishida, Asymptotic behavior and stability of solutions to the exterior convection
problem. Nonlinear Anal. 22 (1994), 895-925.

\bibitem{HY} T. Hishida and Y. Yamada, Global solutions for the heat convection equations in an exterior domain. Tokyo J.
Math. 15 (1992), 135-151.

\bibitem{H97} T. Hishida, On a Class of Stable Steady Flows to the Exterior
Convection Problem. J. Diff. Equs. 141 (1997), 54-85.

\bibitem{H} T. Hishida, Stability of time-dependent motions for fluid-rigid ball
interaction. 	arXiv:2309.13127 [math.AP], Sep. 2023.

\bibitem{L}E. N. Lorenz, Deterministic nonperiodic flow. J. Atmos. Sci. 20 (1963), 130-141.

%\bibitem{K1} H. Kozono, Asymptotic stability of large solutions with large
%perturbation to the Navier-Stokes equations. J.  Funct. Anal. 176 (2000), 153-197.

%\bibitem{K} H. Kozono and T.  Ogawa,  On stability of Navier-Stokes flows in exterior domains. Arch. Rat. Mech.
%Anal. 128 (1994), 1-31.

%\bibitem{Lady69} O. Ladyzhenskaya, The Mathematical Theorey of Viscous Incompressible Flow, Godon and Breach Science Publishers, New York, 1969.
\bibitem{Lady} O. Ladyzhenskaya, The Boundary Value Problems of Mathematical Physics, Springer, New York, 1985.

  % \bibitem{Lions} J. L. Lions, Espaces d'interpolation et domains de puissances fractionaires
%d'operateurs, J. Math. Soc. Japan, 14 (1962), 233-241.

%\bibitem{L}  J. L. Lions,  Quelques m\'{e}thodes de r\'{e}solution de probl\`{e}mes aux limites non lin\'{e}aires. Dunod, Gauthier-Villars, %Paris
%1969.
\bibitem{Miya82} T. Miyakawa,
On nonstationary solutions of the Navier-Stokes
equations in an exterior domain.  Hiroshima Math. J. 12 (1982), 115-140.

\bibitem{Miya1988} T. Miyakawa and H.  Sohr,  On energy inequality, smoothness and large time
behavior in $L^2$ for weak solutions of the Navier-Stokes equations in exterior domains.
Math. Z. 199 (1988), 455-478.

\bibitem{1975Masuda}
K. Masuda, On the stability of incompressible viscous
fluid motions past objects. J. Math. Soc. Japan, Vol. 27, No. 2, 1975.

%\bibitem{interpolation} S Krein,  Linear Differential Eqs. in Banach Spaces. Amer. Math. Soc., Providence, 1972.

\bibitem{R} P. H. Rabinowitz, Existence and nonuniqueness of rectangular solutions of the B\'enard
problem. Arch. Rational Mech. Anal. 29 (1968), 32-57.






\bibitem{Yosida} K. Yosida, Functional Analysis, Sixth Edition, Springer, New York, 1980.

\end{thebibliography}
\end{document}